\documentclass[11pt]{amsart}
  
\input xy
\xyoption{all}
\usepackage{hyperref}
  \usepackage[dvipsnames]{xcolor}
\usepackage{amssymb}
\usepackage{amsmath}
\usepackage{amscd}
\usepackage[english]{babel} 
\usepackage{amssymb,  xfrac, stmaryrd, blkarray, multirow, enumerate,
  verbatim,xifthen}  
\usepackage{latexsym}
\usepackage{bbold}
\usepackage{graphicx, color, pinlabel}    
\begin{document}
 
\newtheorem{lemma}{Lemma}[section]
\newtheorem{prop}[lemma]{Proposition}
\newtheorem{cor}[lemma]{Corollary}
  
\newtheorem{thm}[lemma]{Theorem}
\newtheorem{Mthm}[lemma]{Main Theorem}
\newtheorem{con}{Conjecture}
\newtheorem{claim}{Claim}
\newtheorem{ques}{Question}

\theoremstyle{definition}
  
\newtheorem{rem}[lemma]{Remark}
\newtheorem{rems}[lemma]{Remarks}
\newtheorem{defi}[lemma]{Definition}
\newtheorem{ex}[lemma]{Example}
\newtheorem*{thx*}{Acknowledgement}
\newtheorem{Q}[lemma]{Question}

\newcommand{\bbA}{{\mathbb A}}
\newcommand{\bbB}{{\mathbb B}}
\newcommand{\bbC}{{\mathbb C}}
\newcommand{\bbD}{{\mathbb D}}
\newcommand{\bbE}{{\mathbb E}}
\newcommand{\bbF}{{\mathbb F}}
\newcommand{\bbG}{{\mathbb G}}
\newcommand{\bbH}{{\mathbb H}}
\newcommand{\bbI}{{\mathbb I}}
\newcommand{\bbJ}{{\mathbb J}}
\newcommand{\bbK}{{\mathbb K}}
\newcommand{\bbL}{{\mathbb L}}
\newcommand{\bbM}{{\mathbb M}}
\newcommand{\bbN}{{\mathbb N}}
\newcommand{\bbO}{{\mathbb O}}
\newcommand{\bbP}{{\mathbb P}}
\newcommand{\bbQ}{{\mathbb Q}}
\newcommand{\bbR}{{\mathbb R}}
\newcommand{\bbS}{{\mathbb S}}
\newcommand{\bbT}{{\mathbb T}}
\newcommand{\bbU}{{\mathbb U}}
\newcommand{\bbV}{{\mathbb V}}
\newcommand{\bbW}{{\mathbb W}}
\newcommand{\bbX}{{\mathbb X}}
\newcommand{\bbY}{{\mathbb Y}}
\newcommand{\bbZ}{{\mathbb Z}}

\newcommand{\cA}{{\mathcal A}}
\newcommand{\cB}{{\mathcal B}}
\newcommand{\cC}{{\mathcal C}}
\newcommand{\cD}{{\mathcal D}}
\newcommand{\cE}{{\mathcal E}}
\newcommand{\cF}{{\mathcal F}}
\newcommand{\cG}{{\mathcal G}}
\newcommand{\cH}{{\mathcal H}}
\newcommand{\cI}{{\mathcal I}}
\newcommand{\cJ}{{\mathcal J}}
\newcommand{\cK}{{\mathcal K}}
\newcommand{\cL}{{\mathcal L}}
\newcommand{\cM}{{\mathcal M}}
\newcommand{\cN}{{\mathcal N}}
\newcommand{\cO}{{\mathcal O}}
\newcommand{\cP}{{\mathcal P}}
\newcommand{\cQ}{{\mathcal Q}}
\newcommand{\cR}{{\mathcal R}}
\newcommand{\cS}{{\mathcal S}}
\newcommand{\cT}{{\mathcal T}}
\newcommand{\cU}{{\mathcal U}}
\newcommand{\cV}{{\mathcal V}}
\newcommand{\cW}{{\mathcal W}}
\newcommand{\cX}{{\mathcal X}}
\newcommand{\cY}{{\mathcal Y}}
\newcommand{\cZ}{{\mathcal Z}}

\newcommand{\fra}{{\mathfrak a}}
\newcommand{\frb}{{\mathfrak b}}
\newcommand{\frc}{{\mathfrak c}}
\newcommand{\frd}{{\mathfrak d}}
\newcommand{\fre}{{\mathfrak e}}
\newcommand{\frf}{{\mathfrak f}}
\newcommand{\frg}{{\mathfrak g}}
\newcommand{\frh}{{\mathfrak h}}
\newcommand{\fri}{{\mathfrak i}}
\newcommand{\frj}{{\mathfrak j}}
\newcommand{\frk}{{\mathfrak k}}
\newcommand{\frl}{{\mathfrak l}}
\renewcommand{\frm}{{\mathfrak m}}
\newcommand{\frn}{{\mathfrak n}}
\newcommand{\fro}{{\mathfrak o}}
\newcommand{\frp}{{\mathfrak p}}

\newcommand{\frr}{{\mathfrak r}}
\newcommand{\frs}{{\mathfrak s}}
\newcommand{\frt}{{\mathfrak t}}
\newcommand{\fru}{{\mathfrak u}}
\newcommand{\frv}{{\mathfrak v}}
\newcommand{\frw}{{\mathfrak w}}
\newcommand{\frx}{{\mathfrak x}}
\newcommand{\fry}{{\mathfrak y}}
\newcommand{\frz}{{\mathfrak z}}

\newcommand{\hh}{{\hat h}}
\newcommand{\frA}{{\mathfrak A}}
\newcommand{\frB}{{\mathfrak B}}
\newcommand{\frC}{{\mathfrak C}}
\newcommand{\frD}{{\mathfrak D}}
\newcommand{\frE}{{\mathfrak E}}
\newcommand{\frF}{{\mathfrak F}}
\newcommand{\frG}{{\mathfrak G}}
\newcommand{\frH}{{\mathfrak H}}
\newcommand{\frI}{{\mathfrak I}}
\newcommand{\frJ}{{\mathfrak J}}
\newcommand{\frK}{{\mathfrak K}}
\newcommand{\frL}{{\mathfrak L}}
\newcommand{\frM}{{\mathfrak M}}
\newcommand{\frN}{{\mathfrak N}}
\newcommand{\frO}{{\mathfrak O}}
\newcommand{\frP}{{\mathfrak P}}
\newcommand{\frQ}{{\mathfrak Q}}
\newcommand{\frR}{{\mathfrak R}}
\newcommand{\frS}{{\mathfrak S}}
\newcommand{\frT}{{\mathfrak T}}
\newcommand{\frU}{{\mathfrak U}}
\newcommand{\frV}{{\mathfrak V}}
\newcommand{\frW}{{\mathfrak W}}
\newcommand{\frX}{{\mathfrak X}}
\newcommand{\frY}{{\mathfrak Y}}
\newcommand{\frZ}{{\mathfrak Z}}

\newcommand{\lra}{\longrightarrow}
 \newcommand{\ra}{\rightarrow}
\newcommand{\Ra}{\Rightarrow}
\newcommand{\La}{\Leftarrow}
\newcommand{\Lra}{\Leftrightarrow}
\newcommand{\hra}{\hookrightarrow}
\newcommand{\eps}{\varepsilon}

\newcommand{\heyIndira}[1]{{\color{RoyalBlue}(#1)}}
\newcommand{\addIndira}[1]{{\color{RoyalBlue}#1}}
\newcommand{\heyAlex}[1]{{\color{OliveGreen} (#1)}}
\newcommand{\addAlex}[1]{{\color{OliveGreen} #1}}

\title[Acylindrical hyperbolicity and CAT(0) cube complexes]{A note on the acylindrical hyperbolicity of groups acting on CAT(0) cube complexes}
\author{Indira Chatterji and Alexandre Martin}
\begin{abstract} We study the acylindrical hyperbolicity of  groups acting by isometries on CAT(0) cube complexes, and obtain simple criteria formulated in terms of stabilisers for the action. Namely, we show that a group acting  essentially and  non-elementarily on a finite dimensional  irreducible CAT(0) cube complex  is acylindrically hyperbolic if there exist  two hyperplanes whose stabilisers intersect along a finite subgroup. We  also give further conditions on the geometry of the complex  so that the result holds if we only require the existence of a single pair of points whose stabilisers intersect along a finite subgroup.\end{abstract}

\maketitle

\section{Introduction}

A group is called \emph{acylindrically hyperbolic} if it is not virtually cyclic and admits an acylindrical action with unbounded orbits on a  hyperbolic geodesic metric space. Acylindrically hyperbolic groups  form a large class of groups, introduced by Osin \cite{OsinAcylindricallyHyperbolic}, displaying strong hyperbolic-like features: They encompass mapping class groups of hyperbolic surfaces \cite{BowditchTightGeodesics}, relatively hyperbolic groups \cite{OsinRelativelyHyperbolicGroups}, the plane Cremona group \cite{CantatLamyCremonaNormalSubgroups, LonjouCremonaAcylindricallyHyperbolic}, and many more. The notion of acylindrical hyperbolicity has gathered a lot of interest in recent years due to the strong algebraic and analytical consequences it implies for the group (we refer to \cite{OsinAcylindricallyHyperbolic} and details therein).

In this article, we study the acylindrical hyperbolicity of  groups acting by isometries on CAT(0) cube complexes. Our goal is to obtain simple  acylindrical hyperbolicity criteria for such groups. Recall that an action on a CAT(0) cube $X$ complex is called \emph{essential} if no orbit remains at bounded distance from a half-space of $X$, and that it is called \emph{non-elementary} if it does not admit a finite orbit in $X\cup \partial_\infty X$.

Our first criterion is formulated in terms of hyperplanes stabilisers, and  generalises to finite dimensional CAT(0) cube complexes a criterion due to Minasyan--Osin for actions on simplicial trees \cite{MinasyanOsinTrees}: 
\begin{thm}\label{Main}
 Let $G$ be a group acting essentially and non-elementarily on an irreducible finite-dimensional CAT(0) cube complex. If there exist two hyperplanes whose stabilisers intersect along a finite subgroup, then $G$ is acylindrically hyperbolic.\end{thm}
In the previous theorem, we do not require the two hyperplanes to be disjoint, or even distinct. In particular, the conclusion holds if there exists a  hyperplane of $X$ whose stabiliser is \textit{weakly malnormal}, i.e. it intersects some conjugate along a finite subgroup.

It should be noted that the above theorem does not reduce to the aforementioned criterion of Minasyan--Osin, as groups acting on CAT(0) cube complexes do not virtually act on a simplicial tree a priori. 

Anthony Genevois has also obtained criteria for acylindrical hyperbolicity of a similar flavour, using different tools. In particular, his approach can be used to recover Theorem \ref{Main}, see \cite[Remark 21]{GenevoisAcylindricityHyperplanes}.

We give an application of Theorem \ref{Main} to Artin groups of FC type, suggested to us by Ruth Charney. Artin groups span a large range of groups, and include for instance free groups, braid groups, free abelian groups, as well as many more exotic groups. Artin groups and their subgroups are a rich source of examples and counterexamples of interesting phenomena in geometry and group theory.
\begin{thm}\label{Ruth} Non-virtually cyclic Artin groups of FC type whose underlying Coxeter graphs have diameter at least $3$ are acylindrically hyperbolic.\end{thm}
Artin groups of FC type and the proof of the above theorem will be discussed in Section \ref{Artin}. In the case of Artin groups of FC type whose Coxeter graphs have diameter $1$, i.e. Artin groups of finite type, Calvez and Wiest \cite{CalvezWiest} showed, using different techniques, that the quotient of such groups by their centre is acylindrically hyperbolic.

As a corollary of Theorem \ref{Main}, we obtain the following results for actions satisfying a weak notion of acylindricity or properness: 
\begin{cor}\label{MainAcyl}
Let $G$ be a group acting essentially and non-elementarily on an irreducible finite-dimensional CAT(0) cube complex $X$. Assume that the action is \textit{non-uniformly weakly acylindrical}, that is, there exists a constant $L\geq 0$ such that only finitely many elements  of $G$ fix two points of $X$ at distance at least $L$. Then $G$ is  acylindrically hyperbolic. 
\end{cor}
\begin{cor}\label{MainProper}
A group acting essentially, non-elementarily, and with finite vertex stabilisers on a finite dimensional irreducible CAT(0) cube complex is acylindrically hyperbolic.
\end{cor}
Corollary \ref{MainAcyl} was already known if the complex is in addition assumed to be \textit{hyperbolic} (see \cite{MartinHigmanAcylindrical} for the proof in dimension $2$ and \cite{GenevoisConeOff} for the general case). 
Corollary \ref{MainProper} was already known if  in addition 
 the action is assumed to be metrically proper, by work of Caprace--Sageev on the existence of rank one isometries of  CAT(0) cube complexes \cite{CapraceSageevRankRigidity}. 
Anthony Genevois informed us that he found independently similar results, using different techniques \cite{GenevoisAcylindricityHyperplanes}.

In Theorem \ref{Main} and its corollaries, the essentiality assumption could be weakened to the assumption that the \textit{essential core} is irreducible, a condition that is a priori non-trivial to check. Notice however that the other assumptions cannot be removed. The non-acylindrically hyperbolic group  $\bbZ$ acts properly and essentially, but \textit{elementarily}, on the real line with its standard simplicial structure. The groups of Burger--Mozes provide examples of cocompact lattices in the product of two trees whose associated actions are essential, non-elementary, and without fixed point at infinity, yet these groups are not acylindrically hyperbolic as they are simple \cite{BurgerMozesLatticesTrees}. Thompson's groups $V$ and $T$  act properly, non-elementarily, and without fixed point at infinity on an \textit{infinite dimensional} irreducible CAT(0) cube complex \cite{FarleyBoundaryThompson}, but are not acylindrically hyperbolic as they are simple. \\

We also obtain stronger criteria, that allow us to deduce the acylindrical hyperbolicity of a group from information on the stabilisers of a single pair of points (see Theorem \ref{MainCubesSeparated} for the general statement). For this, we impose further conditions on the complex: We say that a CAT(0) cube complex is \textit{cocompact} if its automorphism group acts cocompactly on it, and we say that it has \textit{no free face} if every non-maximal cube is contained in at least two maximal cubes. We prove the following: 
\begin{thm}\label{MainVert}
Let $G$ be a group acting essentially and non-elementarily on an irreducible finite-dimensional cocompact CAT(0) cube complex with no free face. If there exist two points whose stabilisers intersect along a finite subgroup, then $G$ is acylindrically hyperbolic.
\end{thm}
To show the acylindrical hyperbolicity of a group  $G$ through its action on a geodesic metric space $X$, a useful criterion introduced by Bestvina-Bromberg-Fujiwara \cite[Theorem H]{BestvinaBrombergFujiwara} is to find an infinite order element of the group whose orbits are strongly contracting and which satisfies the so-called \textit{WPD condition}. 
However, this condition, formulated in terms of \textit{coarse} stabilisers of pairs of points, is generally cumbersome to check for actions on non-locally finite spaces. In \cite{MartinTame}, the second author introduced a different criterion involving a weakening of this condition formulated purely in terms of stabilisers of pairs of points, making it much more tractable (see Theorem \ref{prop:General_Criterion_acylindrically_hyperbolic} for the exact formulation). The price to pay is to find group elements satisfying a strengthened notion of contraction of their orbits, called \textit{\"uber-contractions}. The second author showed that such contractions abound for actions on non-locally compact spaces under mild assumptions on the stabilisers of vertices, and used it to show the acylindrical hyperbolicity of the tame automorphism group of an affine quadric threefold, a subgroup of the Cremona group Bir$(\bbP^3(\bbC)) $. In this article, we provide a different way to construct \"uber-contractions for groups acting on CAT(0) cube complexes, this time using the very rich combinatorial geometry of their hyperplanes. This construction relies heavily on the existence of hyperplanes with strong separation properties, called \textit{\"uber-separated hyperplanes}, and introduced by Chatterji--Fern\'os--Iozzi \cite{CFI}. Such hyperplanes were used to prove a superrigidity phenomenon for groups acting non-elementarily on CAT(0) cube complexes.

\begin{thx*} The authors warmly thank Talia Fern\'os for producing the example in Remark \ref{Talia}, Michah Sageev for discussions related to Proposition \ref{cubes_separated}, Carolyn Abbott for noticing Theorem \ref{Ruth} and Ruth Charney for pointing it out to us, along with the explanations of Section \ref{Artin}. We also thank Anthony Genevois for remarks on an early version of this article, Ruth Charney and Rose Morris-Wright for spotting a gap in the proof of Theorem \ref{Ruth},  and the anonymous referee for their interesting comments.

The first author is partially supported by the Institut Universitaire de France (IUF) and the ANR GAMME. The second author is partially supported by the Austrian Science Fund (FWF) grant M1810-N25. This work was completed while the two authors were in residence at MSRI during the Fall 2016 ``Geometric Group Theory'' research program, NSF Grant DMS-1440140. They  thank the organisers for the opportunity to work in such a stimulating environment. \end{thx*}

\section{\"Uber-contractions and  acylindrical hyperbolicity} 

Recall that, given a group $G$ acting on a geodesic metric space $X$, the action is called \textit{acylindrical} if for every $r \geq 0$ there exist constants $L(r), N(r) \geq 0$ such that for every points $x, y$ of $X$ at distance at least $L(r)$, 

$$|\{g\in G | d(x,gx)\leq r,\ d(y,gy)\leq r\}| \leq N(r).$$
A group is \textit{acylindrically hyperbolic} if it is not virtually cyclic and admits an acylindrical action with unbounded orbits on a hyperbolic geodesic metric space. Given a group action on an arbitrary geodesic space, the following criterion was introduced by Bestvina--Bromberg--Fujiwara to prove the acylindrical hyperbolicity of the group: 

\begin{thm}[{\cite[Theorem H]{BestvinaBrombergFujiwara}}]
 Let $G$ be a group acting by isometries on a geodesic metric space $X$. Let $g$ be a group  element of infinite order with quasi-isometrically embedded orbits,  and assume that the following holds: 
 \begin{itemize}
  \item the group element $g$ is \textit{strongly contracting}, that is, there exists a point $x$ of $X$ such that the closest-point projections on the  $\langle g \rangle$-orbit of $x$  of the balls of $X$ that are disjoint from  $\langle g \rangle x$ have uniformly bounded diameter,
  \item the group element $g$ satisfies the \textit{WPD condition}, that is, for every $r \geq 0$ 
  and every point $x$ of $X$, there exists an integer $m$ such that   only finitely many elements $h$ of $G$ satisfy $d(x, hx), d(g^{m}x,hg^{m}x) \leq r$.
 \end{itemize}
Then $G$ is either virtually cyclic or acylindrically hyperbolic.\qed
\end{thm}

The following notion, introduced in \cite{MartinTame}, provides an easier way to prove the acylindrical hyperbolicity of a group.

\begin{defi}\label{def:ubercontraction}
Let $X$ be a geodesic metric space and let  $h$ be  an isometry of $X$ with quasi-isometrically embedded orbits. A {\it system of checkpoints} for $h$ is the data of a finite subset $S$ of $X$, the \textit{checkpoint},  as well as an \textit{error constant} $L \geq 0$ and a quasi-isometry $f: \Lambda:=\bigcup_{i \in \bbZ} h^iS \ra \bbR$ such that we have the following:

Let $x, y$ be points of $X$ and let $x', y'$ be closest-point projections on $\Lambda$ of $x, y$ respectively. For every \textit{checkpoint} $S_i := h^i S, i \in \bbZ$, such that:
\begin{itemize}
	\item $S_i$ \textit{coarsely separates} $x'$ and $y'$ , i.e. $f(x')$ and $f(y')$ lie in different unbounded connected components of $\bbR \setminus f(S_i)$,
	\item $S_i$ is at distance at least $L$ from both $x'$ and $y'$,
\end{itemize}
then every geodesic between $x$ and $y$ meets $S_i$.

A hyperbolic isometry $h$ of $X$ is \textit{\"uber-contracting}, or is an \textit{\"uber-contraction}, if it admits such a system of checkpoints.
\end{defi}

We will be using the following criterion for acylindrical hyperbolicity:

\begin{thm}[Theorem 1.2 of \cite{MartinTame}]\label{prop:General_Criterion_acylindrically_hyperbolic}
	Let $G$ be a group acting by isometries on a geodesic metric space $X$. Let $g \in G$ be an infinite order element such that the following holds: 
	\begin{itemize}
		\item[$(i)$] the group element $g$ is  \"uber-contracting with  respect to a  system of checkpoints $(g^iS)_{i \in \bbZ} $,
		\item[$(ii)$] There exists a constant $m_0$ such that for every point $s \in S$ and every $m \geq m_0$, only finitely many elements of $G$ fix pointwise $s$ and $g^{m}s$.
	\end{itemize}
	Then $G$ is either virtually cyclic or acylindrically hyperbolic.\qed
\end{thm}

Note that condition $(ii)$ of the previous theorem is a considerable weakening of the  WPD condition for $g$. 
 
This weaker condition has the advantage of involving only stabilisers of pairs of points, which makes it much easier to use. 

\section{\"Uber-separated hyperplanes and the proof of Theorem \ref{Main}}

We now recall a few basic facts concerning CAT(0) cube complexes, and more precisely the notions of bridges and \"uber-separated pairs, which are crucial to prove Theorems \ref{Main} and \ref{MainVert}. The missing details and proofs can be found in \cite{CFI}. 

By a slight abuse of notation, we will identify a CAT(0) cube complex with its vertex set endowed with the graph metric coming from its $1$-skeleton. Each hyperplane separates the vertex set into two disjoint components, which we refer to as halfspaces. 

The following notion was introduced by Behrstock--Charney in \cite{Behrstock_Charney}:
\begin{defi}\label{defi:strongly separated} 
Two parallel hyperplanes of a CAT(0) cube complex are said to be {\em strongly separated} if 
 no hyperplane is transverse to both. By the usual abuse of terminology, we say that two halfspaces are strongly separated if the corresponding hyperplanes are. 
\end{defi}

We will need tools introduced by Caprace--Sageev \cite{CapraceSageevRankRigidity}. Recall that a family of $n$ pairwise crossing hyperplanes divides a CAT(0) cube complex into $2^n$ regions called \textit{sectors}.

\begin{lemma}[Caprace--Sageev {\cite[Proposition 5.1, Double-Skewering Lemma, Lemma 5.2]{CapraceSageevRankRigidity}}]\label{CSstuff}
Let $X$ be a finite dimensional irreducible CAT(0) cube complex  and let $ G \ra $ Aut$(X)$ be a group acting essentially and without fixed points in $X \cup \partial_\infty X$. 
\begin{itemize}
\item (Strong Separation Lemma) Let $h_1$ be a halfspace of $X$. Then   there exists a halfspace $h_2$ such that $h_1 \subset h_2$ and  such that $h_1$ and $h_2$ are strongly separated. 
\item (Double-Skewering Lemma) Let $h_1 \subset h_2$ be two nested halfspaces. Then there exists an element $g \in G$ that \textit{double-skewers} $h_1$ and $h_2$, that is, such that $h_1 \subset h_2 \subset gh_1$.
\item (Sector Lemma) Let $\hat h_1$, $\hat h_2$ be two transverse hyperplanes. Then 
we can choose two disjoint hyperplanes $\hat{h}_3$ and $\hat{h}_4$ that are contained in opposite sectors determined by $\hat{h}_1$ and $\hat{h}_2$.\qed
\end{itemize}
\end{lemma}

 We will need a finer notion of strong separation of halfspaces, which is less standard but will be key to our work.

\begin{defi}\label{sss} 
Two strongly separated halfspaces $h_1$ and $h_2$ are said to be an {\em \"uber-separated pair} 
if for  any two halfspaces $k_1, k_2$ with the property that $h_i $ and $ k_i$ are transverse for $i =1, 2$, then  $k_1$ and $ k_2$
are parallel.
We say that two strongly separated hyperplanes are
{\em \"uber-separated} if their halfspaces are. 
 \end{defi}

Note that pairs of \"uber-separated hyperplanes correspond exactly to pairs of hyperplanes at distance at least $4$ in the intersection graph. 

\begin{rem}\label{here} If $h\subset k\subset \ell$ are pairwise strongly separated halfspaces, then
$h$ and $\ell$ are \"uber-separated.\end{rem}

Notice that \"uber-separated pairs are in particular strongly separated and hence they do not exist in the reducible case by
\cite[Proposition~5.1]{CapraceSageevRankRigidity}. The existence of \"uber-separated hyperplanes is a consequence of the Double-Skewering Lemma \ref{CSstuff}.
We will need the following lemma, which is a direct consequence of the  Double-Skewering Lemma \ref{CSstuff} and \cite[Lemma 2.14]{CFI}:

\begin{lemma}\label{lem:ss}Let $X$ be a finite dimensional irreducible CAT(0) cube complex 
and $G\to{\rm Aut}(X)$ a group acting essentially and non-elementarily.  
Given any two parallel hyperplanes $\hat{h}_1$ and $\hat{h}_2$, there exists $g\in G$ that double-skewers $\hat{h}_1$ and $\hat{h}_2$ and such that 
$\hat{h}_1$ and $g\hat{h}_1$ are \"uber-separated. \qed
\end{lemma}

For two points $x, y$ of $X$, we denote by $\mathcal{I}(x,y)$ the interval between $x$ and $y$, that is, the union of all the geodesics between $x$ and $y$. Recall that intervals are finite, as they only depend on the (finite) set of hyperplanes separating the given pair of points.

\begin{defi}
Let $h_1\subset h_2$ be a nested pair of halfspaces. 
The \textit{(combinatorial) bridge} between $\hat h_1$ and $\hat h_2$, denoted $b(\hat h_1,\hat h_2)$, is the union of all the geodesics between points $x_1 \in  h_1$ and $x_2 \in h_2^*$ minimizing the distance between $h_1$ and $h_2^*$.
\end{defi}

\begin{lemma}[Lemma 2.18 and 2.24 {\cite{CFI}}] \label{uber} If $h_1\subset h_2$ be a pair of nested halfspaces, strongly separated, there exists a unique pair of points of $h_1\times h_2^*$, called the gates of the bridge, that minimizes the distance between $h_1$ and $h_2^*$, i.e. there exist points $x_1\in h_1$ and $x_2\in h_2^*$ such that
$$b(\hat h_1,\hat h_2)=\mathcal{I}(x_1,x_2).$$ 
In particular, the bridge between two  strongly separated hyperplanes is finite. 
Moreover, the following holds true for any $y_1\in h_1$ and $y_2\in h_2^*$.

$$d(y_1,y_2)=d(y_1,x_1)+d(x_1,x_2)+d(x_2,y_2).  $$\qed
\end{lemma}
 
 The following is a very important feature of \"uber-separated pairs.
 
\begin{lemma}[Proof of Lemma 3.5 of \cite{CFI}]\label{key} Let $h_1\subset h_2$ be an \"uber-separated pair of halfspaces, $x\in h_1$ and $y\in h_2^*$. Then every geodesic between $x$ and $y$ meets the bridge  $b(\hat{h}_1,\hat{h}_2)$. \qed
\end{lemma}

The following lemma explains the relationship between \"uber-separated pairs and \"uber-contractions.

\begin{lemma}\label{gates}Let  $h_1\subset h_2$ be an \"uber-separated pair of halfspaces, and $g\in G$ an element that double-skewers  $h_1$ and $h_2$. Then $g$ is an \"uber-contraction.\end{lemma}

\begin{proof}
Let $B$ denote the bridge between $h_1$ and $gh_1$.  We will show that the collection of the $g$-translates of $B$ forms a system of checkpoints. 
Notice that since the bridge is an interval, it is finite even though the complex is not assumed to be locally finite, hence all the checkpoints $S_n:= g^nB$ are finite. Set $\Lambda:= \bigcup_{n \in \bbZ}S_n$. Let $Y:= h_1^* \cap gh_1$ and $Y_n:=g^n Y$ for every $n \in \bbZ$. 

For every point $x \in Y_n$, we have that its closest-point projections on $\Lambda$ are in $S_{n-1} \cup S_n \cup S_{n+1}$: Indeed, since $h_1$ and $gh_1$ are \"uber-separated, it follows from Lemma \ref{key} that a geodesic between $x$ and a point $x' \in Y_{n'}$ with $n'\geq n+2$ ($n' \leq n-2$ respectively) meets the bridge $S_{n+1}$ ($S_{n-1}$ respectively). Moreover, for every $x \in X$ which projects on $\Lambda$ to a point of $S_n$, we have that $x \in Y_{n-1} \cup Y_n \cup Y_{n+1}$ for the same reasons. Thus, for $x,y\in X$ which project on $\Lambda$ to points $x' \in S_n$ and $y'\in S_m$ respectively with $m \geq n+4$, it follows that $x$ and $y$ are separated by the hyperplanes $g^{n+2}\hh_1, \ldots, g^{m-1}\hh_1$, hence every  geodesic  between $x$ and $y$ meets each checkpoint $S_{n+2}, \ldots, S_{m-2}$ by Lemma \ref{key}.

By a result of Haglund \cite[Theorem 1.4]{HaglundQICyclicCubeComplexes}, $g$ admits a combinatorial axis $\Lambda_g$. As $\Lambda$ and $\Lambda_g$ stay at bounded distance from one another, it follows easily from the discussion of the previous paragraph that $\Lambda_g$ is contained in $\Lambda$. Since $B$ is finite, the closest-point projection $\Lambda \ra \Lambda_g$ yields a quasi-isometry $\Lambda \ra \bbR$, and it is now straightforward to check that $g$ is an \"uber-contraction.
\end{proof}

\begin{proof}[Proof of Theorem \ref{Main}]Let $G$ be a group acting essentially and non-elementarily on an irreducible finite-dimensional CAT(0) cube complex $X$ and assume that there exist two hyperplanes whose stabilisers intersect along a finite subgroup. By Proposition \ref{prop:General_Criterion_acylindrically_hyperbolic},  it is enough to construct an element of $G$ satisfying conditions  $(i)$ and $(ii)$ of Proposition \ref{prop:General_Criterion_acylindrically_hyperbolic}.

Let $\hat{h}_1$ and $\hat{h}_2$ be two hyperplanes whose stabilisers intersect along a finite subgroup. 
We start by showing that we can assume that $\hat{h}_1$ and $\hat{h}_2$ are disjoint. 
By the Sector Lemma \ref{CSstuff}, we choose two disjoint hyperplanes $\hat{h}_3$ and $\hat{h}_4$ that are contained in opposite sectors determined by $\hat{h}_1$ and $\hat{h}_2$. Up to applying the Strong Separation Lemma \ref{CSstuff}, we can further assume that $\hat{h}_3$ and $\hat{h}_4$ are strongly separated. Let $H:= \mbox{Stab}(\hat{h}_3) \cap \mbox{Stab}(\hat{h}_4)$. Then $H$ stabilises the bridge between $\hat{h}_3$ and $\hat{h}_4$, which is a single interval by Lemma \ref{uber}  since $\hat{h}_3$ and $\hat{h}_4$ are strongly separated. As intervals are finite, $H$ virtually fixes a geodesic between $\hat{h}_3$ and $\hat{h}_4$. As $\hat{h}_3$ and $\hat{h}_4$ are in opposite sectors determined by $\hat{h}_1$ and $\hat{h}_2$, such a geodesic crosses both  $\hat{h}_1$ and $\hat{h}_2$. Thus, $H$ is virtually contained in $\mbox{Stab}(\hat{h}_1) \cap \mbox{Stab}(\hat{h}_2)$, hence $H$ is finite, which is what we wanted.

Thus, let $\hat{h}_1$ and $\hat{h}_2$ be two disjoint hyperplanes whose stabilisers intersect along a finite subgroup. By Lemma \ref{lem:ss}, we choose an element $g$ of $G$ that double-skewers $\hat{h}_1$ and $\hat{h}_2$, and such that $\hat{h}_1$ and $g\hat{h}_1$ are \"uber-separated. Then according to Lemma \ref{gates} the element $g$ is an \"uber-contraction, proving $(i)$. Let $x$ be a point of the bridge between $\hat{h}_1$ and $\hat{h}_2$, and choose a geodesic $\gamma$ between $x$ and $g^2x$. We have that $\gamma$ crosses $g\hat{h}_1$ and $g\hat{h}_2$. Since intervals in a finite dimensional CAT(0) cube complex are finite, a subgroup fixing both $x$ and $g^2x$ virtually fixes $\gamma$ pointwise. It follows that  $\mbox{Stab}(x)\cap \mbox{Stab}(g^2x)$ is virtually contained in a conjugate of $\mbox{Stab}(\hat{h}_1) \cap \mbox{Stab}(\hat{h}_2)$, which is finite. This proves $(ii)$ , and Proposition \ref{prop:General_Criterion_acylindrically_hyperbolic} implies that $G$ is either virtually cyclic or acylindrically hyperbolic. The action being non-elementarily, the virtually cyclic case is automatically ruled out, which concludes the proof.
\end{proof}

\begin{proof}[Proof of Corollaries \ref{MainAcyl} and \ref{MainProper}]
Since Corollary \ref{MainProper} is a direct consequence of Corollary \ref{MainAcyl}, let us assume that the group $G$ acts essentially, non-elementarily, and non-uniformly weakly acylindrically (with a constant $L$ as in the statement) on the irreducible finite-dimensional CAT(0) cube complex $X$. By the Strong Separation Lemma \ref{CSstuff}, choose two disjoint hyperplanes $\hat h_1$ and $\hat h_2$. This implies that the combinatorial bridge between $\hat h_1$ and $\hat h_2$ is not reduced to a point. Moreover, the subgroup $\mbox{Stab}(\hat h_1)\cap \mbox{Stab}(\hat h_2)$ is virtually contained in the pointwise stabiliser of the finite bridge between $\hat h_1$ and $\hat h_2$ by Lemma \ref{uber}.
 
By the Double-Skewering Lemma \ref{CSstuff}, choose a group element $g$ that skewers both $\hat h_1$ and $\hat h_2$. For $n$ large enough, $\hat h_1$ and $g^n\hat h_1$ are \"uber-separated, and the distance between $\hat h_1$ and $g^n \hat h_1$ becomes greater than $L$ by Lemma \ref{uber}.  In particular, $\mbox{Stab}(\hat h_1)\cap \mbox{Stab}(g^n \hat h_1)$ virtually fixes a path of length $L$, hence $\mbox{Stab}(\hat h_1)\cap \mbox{Stab}(g^n \hat h_1)$ is finite by weak acylindricity. Corollary \ref{MainAcyl} now follows from Theorem \ref{Main}.
\end{proof}

Theorem \ref{Main} allows for a very simple geometric proof of the acylindrical hyperbolicity of certain groups: 

\begin{ex}\label{Higman}
The Higman group on $n \geq 4$ generators, defined by the following presentation: 
$$H_n:= \langle a_i, i \in \bbZ / n \bbZ ~|~ a_ia_{i+1}a_i^{-1}= a_{i+1}^{2}, i \in \bbZ / n \bbZ  \rangle,$$
was proved to be acylindrically hyperbolic by Minasyan--Osin \cite{MinasyanOsinTrees}, by means of its action on the Bass--Serre tree associated to some splitting. The Higman also acts cocompactly, essentially, and non-elementarily  on a CAT(0) square complex associated to its standard presentation. That CAT(0) square complex is irreducible, as links of vertices are  easily shown not to be complete bipartite graphs (see \cite[Corollary 4.6]{MartinHigmanCubical}). Since square stabilisers are trivial by construction, the stabilisers of two crossing hyperplanes intersect  along a finite (actually, trivial) subgroup. Thus, Theorem \ref{Main} applies.
\end{ex}

\section{Proof of Theorem \ref{MainVert}}\label{Applications}

\begin{defi}
Let $X$ be CAT(0) cube complex and $C, C'$ two cubes of $X$. We say that $C$ and $C'$ \textit{separate a pair of hyperplanes} if there exist two hyperplanes $\hh, \hh'$ of $X$ such that each hyperplane defined by an edge of $C \cup C'$ separates $\hh$ and $\hh'$.
\end{defi}

Theorem \ref{MainVert} will be a consequence of the following more general result: 

\begin{thm}\label{MainCubesSeparated}
Let $G$ be a group acting essentially and non-elementarily on a finite dimensional irreducible CAT(0) cube complex. Assume that there exist  two maximal cubes $C$, $C'$  of $X$ whose stabilisers intersect along a finite subgroup, and such that $C$ and $C'$ separate a pair of hyperplanes. Then $G$ is acylindrically hyperbolic.
\end{thm}

\begin{proof}
Let $\cH_{C, C'}$ be the set of hyperplanes defined by the edges of $C \cup C'$. Since $C$ and $C'$ are maximal cubes, we have  that 
 $ \bigcap_{\hat k \in \cH_{C,C'}} \mbox{Stab}(\hat k)$   is contained in $\mbox{Stab}(C)\cap \mbox{Stab}(C')$, 
and it follows that $\bigcap_{\hat k \in \cH_{C,C'}} \mbox{Stab}(\hat k)$ is finite. 

By assumption, choose two disjoint halfspaces $h, h'$ separated  by each hyperplane of $\cH_{C, C'}$. Up to applying the Strong Separation Lemma \ref{CSstuff}, we can further assume that $h$ and $h'$ are strongly separated. In particular, since the bridge between $h $ and $h'$ is finite by Lemma \ref{uber}, it follows that $\mbox{Stab}(\hat h)\cap \mbox{Stab}(\hat h')$ is virtually contained in the pointwise stabiliser of a geodesic between the two gates of the bridge $b(\hat h, \hat h')$. As such a geodesic crosses each hyperplane of $\cH_{C, C'}$ by construction of $h, h'$, it follows that $\mbox{Stab}(\hat h)\cap \mbox{Stab}(\hat h')$ is virtually contained in $\bigcap_{\hat k \in \cH_{C,C'}} \mbox{Stab}(\hat k)$, which is finite by the above argument. Thus, $\mbox{Stab}(h)\cap \mbox{Stab}(h')$  is finite, and we conclude with  Theorem \ref{Main}. 
\end{proof}

The following proposition gives a class of examples of CAT(0) cube complexes where each pair of cubes separates a pair of hyperplanes: 

\begin{prop}\label{cubes_separated}
Let $X$ be an irreducible CAT(0) cube complex without free face, and such that Aut$(X)$ acts cocompactly on $X$. Then each pair of cubes separates a pair of hyperplanes.
\end{prop}

\begin{rem}\label{Talia}The following example, due to Talia Fern\'os, shows that the no-free-face assumption is necessary in Proposition \ref{cubes_separated}. Take a 3-dimensional cube $[0,1]^3$ and glue an edge to each of the vertices $(1,0,0),(0,1,0),(0,0,1)$ and $(1,1,1)$ to get a spiked cube. Then glue infinitely many of these spiked cubes in a tree-like way, to obtain a cocompact CAT(0) cube complex $X$ which is quasi-isometric to a tree. Then none of the 3-cubes of that complex is separating a pair of hyperplanes. Indeed, each 3-cube has three hyperplanes defining it, hence it defines eight sectors, out of which four contain hyperplanes and four are reduced to a single point, but among the sectors containing hyperplanes, no two are opposite.\end{rem}

Before proving  Proposition \ref{cubes_separated}, we start by a simple observation: 

\begin{lemma}\label{extend}
A CAT(0) cube complex with no free face is geodesically complete, that is, every finite geodesic can be extended to a bi-infinite geodesic.
\end{lemma}

\begin{proof}
Let $\gamma$ be a finite geodesic defined by a sequence $e_1, \ldots, e_n$ of edges of $X$ and let $v$ be the terminal vertex of that finite geodesic. We will show that we can extend it by one edge. Let $E$ denote the set of edges of $X$, containing $v$ and such that $\gamma$ followed by $e\in E$ is not a geodesic. Then, every $e\in E$ has to be parallel to one of the edges $e_i$ defining $\gamma$, that is, every $e\in E$ defines a hyperplane $\hat{h}_e$ crossed by $\gamma$. Moreover, the map $e\mapsto\hat{h}_e$ is injective as two adjacent edges cannot belong to the same hyperplane, and any two $\hat{h}_e$, $\hat{h}_{e'}$ intersect. Indeed, for an edge $e$ to belong to $E$ it means that $\gamma$ has been traveling on the carrier of the hyperplane defined by some $e_i$ after having crossed $e_i$, and that can be done simultaneously for several hyperplanes only when they cross each other. Hence $E$ defines a cube in $X$, and since there are no free faces there is an edge $e$ containing $v$ and that does not belong to $E$, allowing us to extend by one edge the geodesic $\gamma$.
 \end{proof}
We will also need the following strengthening of the Sector Lemma:
\begin{lemma}[Strong Sector Lemma]\label{strongsector}
Let $X$ be a CAT(0) cube complex with no free face, and assume that the automorphism group of $X$ acts cocompactly on $X$. Then each sector determined by a finite family of pairwise crossing hyperplanes contains a hyperplane.
\end{lemma}
\begin{proof}
We prove the result by induction on the number $n \geq 2$ of pairwise crossing hyperplanes. For $n=2$, the result follows from \cite[Proposition 3.3]{SageevLecturesCCC}, since a CAT(0) cube complex with no free face is geodesically complete by Lemma \ref{extend}. Let $\hh_1, \ldots, \hh_n$ be a family of pairwise crossing hyperplanes, and let $h_i$ be a halfspace associated to $\hh_i$ for every $i$.  We want to construct a hyperplane contained in $\bigcap_i h_i$. 

By Helly's theorem, the family $(\hh_i \cap \hh_1)_{i\neq 1}$ defines a family of pairwise crossing hyperplanes of $\hh_1$. Note that $\hh_1$ also satisfies the property of having no free face, and the action of $\mbox{Stab}(\hh_1)$ on $\hh_1$ is cocompact, as the same holds for the action of $\mbox{Aut}(X)$ on $X$. Thus, one can apply the induction hypothesis to find a hyperplane $\hh'$ of $\hh_1$ contained in the sector $\bigcap_{i\neq 1} (\hh_1 \cap h_i)$. This defines a hyperplane of $X$, which we denote $\hh$. Since $\hh'$ is disjoint from the $\hh_1 \cap \hh_i$ for $i \neq 1$, it follows from Helly's theorem that $\hh$ is disjoint from the $\hh_i$ for $i \neq 1$. Thus, let $h$ be the halfspace of $\hh$ contained in $\bigcap_{i \neq 1} h_i$. Since $\hh$ and $\hh_1$ cross,  we  can choose by the induction hypothesis  a hyperplane contained in the sector  $h \cap h_1$, hence in $\bigcap_i h_i$. 
\end{proof}

\begin{proof}[Proof of Proposition \ref{cubes_separated}]
It is enough to prove the proposition when $C$ and $C'$ are maximal.  Choose $x$ and $x'$ vertices of $C$,  $C'$ respectively that maximize the distance between vertices of $C$ and $C'$. Let $\cH_C$, $\cH_{C'}$ be the family of hyperplanes defined by an edge of $C$, $C'$ respectively. By the Strong Sector Lemma \ref{strongsector}, each sector determined by $\cH_C$ or $\cH_{C'}$ contains a hyperplane. Thus, choose a hyperplane $\hat h$ ($\hat h'$ respectively) in the unique sector determined by $\cH_C$ ($\cH_{C'}$ respectively) that contains $x$ ($x'$ respectively).  By construction of $x$ and $x'$, we have that $x$ and $x'$ are in opposite sectors defined by $\cH_C$, and in opposite sectors determined by $\cH_{C'}$.  In particular, every hyperplane of $\cH_{C, C'}$ separates $x$ and $x'$, hence separates $\hat h$ and  $\hat h'$.  
\end{proof}

\begin{proof}[Proof of Theorem \ref{MainVert}]
Theorem \ref{MainVert} is a direct consequence of Theorem \ref{MainCubesSeparated} and Proposition \ref{cubes_separated}.
\end{proof}

Theorem \ref{MainVert} allows for a very simple geometric proof of the acylindrical hyperbolicity of certain groups:

\begin{ex}\label{Tame}
The group of tame automorphisms of an affine quadric threefold, a subgroup of the $3$-dimensional Cremona group Bir$(\bbP^3(\bbC))$, acts cocompactly, essentially and non-elementarily on a hyperbolic CAT(0) cube complex without free face \cite{LamyCremonaSquareComplexes}. The second author showed that there exist two cubes whose stabilisers intersect along a finite subgroup \cite[Proof of Theorem 3.1]{MartinTame}, hence Theorem \ref{MainVert} applies.
\end{ex}
\section{Artin groups of type FC}\label{Artin}
We now give an application of our results to the class of Artin groups of FC type, studied by Charney and Davis in \cite{CD} and that we now describe. Recall that a \emph{Coxeter graph} is a finite, simplicial graph $\Gamma$ with vertex set $S$ and edges labeled by integers greater or equal to $2$.  The label of the edge connecting two vertices $s$ and $t$ is denoted $m(s, t)$, and we set $m(s, t) =\infty$ if $s$ and $t$ are not connected by an edge. The \emph{Artin group associated to a Coxeter graph} $\Gamma$ is the group given by the presentation: 
$$A =\left<S  | \underbrace{sts\dots}_{m(s,t)} = \underbrace{tst\dots}_{m(s,t)}|s,t \hbox{ connected by an edge with label }m(s,t)\right>.$$
Adding the extra relations $s^2 = 1$ for all $s\in S$, we obtain a Coxeter group $W$ as a quotient of $A$. We say that $A$ is \emph{finite type} if the associated Coxeter group $W$ is finite.
It was shown in \cite{Lek} that, if $T\subseteq S$, the subgroup $A_T$  generated by $T$ is isomorphic to the Artin  group associated to the full subgraph of $\Gamma$ spanned by $T$. Such subgroups are called \emph{special subgroups} of $A$. 
Following Charney--Davis, we say that an Artin group is of \emph{FC type} if every complete subgraph of the Coxeter graph $\Gamma$ generates a special subgroup of finite type.

Given an Artin group $A$, the \emph{Deligne complex} ${\mathcal D}_A$ is the cubical complex defined as follows: 
Vertices of $\cD_A$ correspond to cosets $aA_T$, where $a \in A$ and $ T\subseteq S$. Note that we allow $T =\emptyset$, in which case $aA_T = \{a\}$.  The $1$-skeleton $\cD_A^1$ of $\cD_A$ is obtained by putting an edge between cosets of the form $aA_T$ and $aA_{T'}$, when $T'=T\cup\{t'\}$ for some $t'\in S\setminus T$.  In particular, each edge of $\cD_A$ is labeled by an element of $S$.  Finally, $\cD_A$ is obtained by ``filling the cubes'', that is, by glueing a $k$-cube whenever $\cD_A^1$ contains a subgraph isomorphic to the $1$-skeleton of a $k$-cube. 

In \cite{CD} Theorem 4.3.5 Charney and Davis show that the Deligne complex $\cD_A$ of an Artin group $A$  is a CAT(0) cube complex if and only if $A$ is of FC type. Edges and hyperplanes in this CAT(0) cube complex are labeled by elements of $S$. Moreover, each hyperplane is a translate of a hyperplane $\hat{h}_s$  for some $s\in S$, where  $\hat{h}_s$ denotes the hyperplane defined by the edge between $\{1\}$ and $A_{\{s\}}$. It is straightforward to check that the stabiliser of $\hat{h}_s$ is the special subgroup $A_{lk(s)}$ of $A$ determined by the link of the vertex $s$ in $\Gamma$.

The \emph{dimension} of an Artin group is the maximum cardinality of a subset $T\subseteq S$ such that $A_T$ is finite type, and it is equal to the dimension of  ${\mathcal D}_A$. In particular, ${\mathcal D}_A$ is finite dimensional since the graph $\Gamma$ is finite. The Artin group $A$ acts by left multiplication on the aforementioned cosets, and hence acts by isometries on ${\mathcal D}_A$. Since $\Gamma$ is finite, the action is cocompact. The stabilizer of a vertex $aA_T$  of $\cD_A$ is the subgroup $aA_Ta^{-1}$. In particular, the action is not proper if $\Gamma$ is non-empty.

The action satisfies the following:
\begin{prop}\label{DeligneAction}
Let $A$ be an Artin group of type FC associated to a Coxeter graph $\Gamma$ of diameter at least 3. Then  the Deligne complex $\cD_A$ is irreducible and the action of $A$ on $\cD_A$ is essential and non-elementary.
\end{prop}
The proof will take most of this section. We can assume that $\Gamma$ is connected for otherwise $A$ is a free product, $\cD_A$ has a structure of a tree of spaces, and the result follows. We check separately the irreducibility of the Deligne complex, the essentiality of the action, and the non-elementarity of the action.

\begin{lemma}\label{Deligne_irreducible}
The CAT(0) cube complex $\cD_A $ is irreducible.
\end{lemma}

 \begin{proof} Consider the vertex $\{1\}$ of $\cD_A$, where $1$ denotes the identity element. The labeling of the edges of $\cD_A$ yields a surjective map $lk(\{1\}) \ra \Gamma$. As $\Gamma$ has diameter at least $3$, so does $lk(\{1\})$, and it follows that $lk(\{1\})$ is not a join, hence $\cD_A$ is an irreducible CAT(0) cube complex.
\end{proof}

\begin{lemma}\label{Deligne_essential}
The action of $A$ on $\cD_A $ is essential.
\end{lemma}

 \begin{proof}Since the action is cocompact, it is enough to show that each hyperplane is essential, that is, no half-space is contained in a neighbourhood of the other halfspace. As each hyperplane is a  translate of some $\hh_s$, it is enough to show that this is the case for hyperplanes of the form $\hh_s$, where $s \in S$.  Let $s_0$ be a vertex of $\Gamma$. Since $\Gamma$ is connected and has diameter at least $3$, we can find  two distinct vertices $s_1, s_2$ of $\Gamma$ such that $s_0, s_1, s_2$ defines a geodesic of $\Gamma$. Then the hyperplane $\hh_{s_1}$ is in particular stabilized by $A_{\{s_0, s_2\}}$. Thus, to show that $\hh_{s_0}$ is essential, it is enough to show that $\hh_{s_1}$ is unbounded and crosses $\hh_{s_0}$. Let $C_{s_0,s_1}, C_{s_1, s_2}$ be the squares of $\cD_A$ containing the vertices $\{1\}, A_{\{s_0\}}, A_{\{s_1\}}$ and $\{1\}, A_{\{s_1\}}, A_{\{s_2\}}$ respectively. Notice that the edge between $A_{\{s_0\}}$ and $A_{\{s_0, s_1\}}$ has stabiliser $A_{\{s_0\}}$, and the edge between $A_{\{s_2\}}$ and $A_{\{s_1, s_2\}}$ has stabiliser $A_{\{s_2\}}$. Thus, the $A_{\{s_0, s_2\}}$-orbit $Y$ of $C_{s_0,s_1}\cup C_{s_1, s_2}$ defines a subcomplex of $\cD_A$ that is convex for the CAT(0) metric and quasi-isometric to a tree. Moreover, every point of $Y$ contained in a half-space of $\hh_{s_0}$ projects on the other half-space to a point of $C_{s_0,s_1}\cup C_{s_1, s_2}$. In particular, each half-space of $\hh_{s_0}$ contains points of $Y$ arbitrarily far away from the other half-space, hence $\hh_{s_0}$ is essential. 
\end{proof}

\begin{lemma}\label{finite_orbit}
The action of $A$ on $\cD_A $ is non-elementary.
\end{lemma}

 The proof of this lemma requires some preliminary work. Since $\Gamma$ has diameter at least $3$,  let $s_0, s_1, s_2, s_3$ be a geodesic of $\Gamma$.

\begin{lemma} Let $g:=s_0s_3$.  Then $g$ is a hyperbolic element and admits  an axis $\Lambda_g$ (for the CAT(0) metric) that is  a reunion of geodesic segments such that two consecutive segments make an angle strictly greater than $\pi$ (for the angular distance on the link).
\end{lemma}

\begin{proof}
For $i = 0 $ or $3$, we denote by $e_i$ the edge between the vertices $\{1\}$  and $A_{\{s_i\}}$ of $\cD_A$. Let $Y:= s_0^{-1}e_0\cup  e_0 \cup e_3 \cup s_3e_3$. Then $\Lambda_g:= \bigcup_{n\in \bbZ}g^n Y$ is a CAT(0) geodesic, and an axis for $g$. Indeed, we have the following properties of angles between consecutive edges: 
\begin{itemize}
\item The angle (for the angular distance on the link) between $e_0$ and $e_3$ is $3\pi/2>\pi$ by construction of  $s_0, s_1, s_2, s_3$.
\item The angle between $e_0$ and $s_0^{-1}e_0$ (between $e_3$ and $s_3e_3$ respectively) is $\pi$: Indeed, the angle is at most $\pi$ by construction, and if the angle were $\pi/2$, then since $\cD_A$ is a CAT(0) cube complex $e_i$ and $s_ie_i$ would be two adjacent edges of a $3$-cube with label $e_i$, which is impossible by construction of $\cD_A$.
\end{itemize}
Thus, $\Lambda_g$ is a local geodesic, hence a global geodesic. Moreover, $\Lambda_g$ is clearly invariant under the action of $\langle g \rangle$, and the angle made by $\Lambda_g$ at every $\langle g \rangle$-translate of the vertex $\{1\}\in \cD_A$ is $3\pi/2$. 
\end{proof}

To show that $g$ is a rank-one element, we  need the following modified version of the Flat Strip Theorem:

\begin{lemma}[Flat ``Half-strip'' Theorem]\label{Half_Strip}
Let $Y$ be a CAT(0) space, let $\Lambda$ be a geodesic line, and let $h$ be an isometry of $Y$ preserving $\Lambda$. Let $y \in \partial Y \setminus \partial \Lambda$ and let $\gamma$ be a geodesic ray from a point $x$ of $\Lambda$ to $y$, which meets $\Lambda$ in exactly one point. If $\gamma$ and $h\gamma$ are asymptotic, then the convex hull of $\gamma \cup h\gamma$ isometrically embeds in $\bbR^2$ with its standard CAT(0) metric. 
\end{lemma}

\begin{proof}
Let us construct a `double' of $Y$ as follows. Let $o \in \Lambda$ be the midpoint between $x$ and $hx$. Such a choice  allows to define (uniquely) a reflection $\psi$ of $\Lambda$ across $o$ which is an isometry of $\Lambda$. We then define 
$$Y' = (Y \sqcup Y)/\psi,$$
i.e. the space obtained from two copies of $Y$ by identifying the two copies of $\Lambda$ using the reflection $\psi$. As a convention, we denote by $Y_1$ and $Y_2$ these copies, we use the subscript $\cdots_i$ to indicate to  which copy of $Y$ the object belongs, and we identify $Y$ with the subspace $Y_1$ of $Y'$.  The space $Y'$ is again a CAT(0) space (as obtained from two CAT(0) spaces by identifying two convex subspaces along an isometry). Note that $\gamma\subset Y$ is a subray of  $\gamma_1 \cup (h\gamma)_2$ and $h\gamma\subset Y$ is a subray of  $(h\gamma)_1\cup \gamma_2$. Moreover, we have constructed $Y'$ so that  $\gamma_1 \cup (h\gamma)_2$ and  $(h\gamma)_1\cup \gamma_2$ are local geodesics of $Y'$, hence global  geodesics of $Y'$. As these geodesic lines are asymptotic by construction, it follows that their geodesic hull is a flat strip, by the Flat Strip Theorem \cite[Theorem II$.2.13$]{BridsonHaefliger}. In particular, the geodesic hull of $\gamma \cup h\gamma$ in $Y$ isometrically embeds in $\bbR^2$.
\end{proof}

We can now prove:

\begin{lemma}
The only points of $\partial \cD_A$ fixed by $g$ are its two limit points $g^{+\infty}, g^{-\infty} \in \partial \Lambda_g$.  
\end{lemma}

\begin{proof} 
Let $z\in \partial \cD_A \setminus \{g^{+\infty}, g^{-\infty}\}$  and let $\gamma'$ be a geodesic ray from $\Lambda_g$ to $z$ that meets $\Lambda_g$ in exactly one point. If $gz=z$, then $\gamma'$ and $g^2\gamma'$ are asymptotic geodesic rays, hence the convex hull $H$ of $\gamma' \cup g^2\gamma'$ isometrically embeds in $\bbR^2$  by the Flat Half-strip Theorem \ref{Half_Strip}. But  $\gamma$ and $ g^2\gamma'$ both meet $\Lambda_g$ in exactly one point, and by construction $H$ contains two geodesic subsegments of $\Lambda_g$ which make an angle $3\pi/2$ for the angular distance on the link. This yields the desired contradiction.
\end{proof}

\begin{proof}[Proof of Lemma \ref{finite_orbit}]
First notice that one could have proved the above lemma for the element $h:= s_0^2s_3$ by applying exactly the same reasoning. If $A$ were to admit a finite orbit at infinity, then some power of $g$ and some power of $h$ would fix a common point at infinity. But the axes of $g$ and $h$ cannot be asymptotic for otherwise, being already distinct,  $\Lambda_g$ and $\Lambda_h$ would bound a flat strip of positive width by the Flat Strip Theorem \ref{Half_Strip}, and the same reasoning as above would  yield a contradiction.
\end{proof}

\begin{proof}[Proof of Proposition \ref{DeligneAction}]
This is a direct consequences of Lemmas \ref{Deligne_irreducible}, \ref{Deligne_essential}, and \ref{finite_orbit}.
\end{proof}
We are now ready to apply Theorem \ref{Main} to the action of $A$ on $\cD_A$ in order to complete the proof of Theorem \ref{Ruth}.
\begin{proof}[Proof of Theorem \ref{Ruth}]
It is enough to consider the case where $\Gamma$ is connected, for otherwise the group is a free product. According to the previous lemma, the action of $A$ on its Deligne complex $\cD_A$ is essential and non-elementary. In order to apply Theorem \ref{Main} to the action of $A$ on $\cD_A$,  
we choose two vertices $s, t$ of $S$ with disjoint links, which is possible since $\Gamma$ has diameter at least $3$. It follows from the aforementioned result of \cite{Lek} that $A_{lk(s)} \cap A_{lk(t)}=A_\emptyset=\{1\}$. Thus, the hyperplanes $\hh_s$ and $\hh_t$ have stabilisers which intersect trivially, hence Theorem \ref{Main} applies and $A$ is acylindrically hyperbolic. 
\end{proof}

\bibliographystyle{alpha}
\bibliography{Higman_Cubical}

\bigskip
  \footnotesize

 \textsc{Laboratoire de Math\'ematiques J.A. Dieudonn\'e, Universit\'e de Nice Sophia Antipolis, Parc Valrose, 06108 Nice Cedex 02, France} \par\nopagebreak
 \texttt{\href{mailto:indira.chatterji@math.cnrs.fr}{indira.chatterji@math.cnrs.fr}} 
 
  \textsc{Department of Mathematics and the Maxwell Institute for Mathematical Sciences, Heriot-Watt University, Riccarton, EH14 4AS Edinburgh, United Kingdom}\par\nopagebreak
\texttt{\href{mailto:alexandre.martin@hw.ac.uk}{alexandre.martin@hw.ac.uk}}

\end{document}